\documentclass{amsart}
\usepackage[latin1]{inputenc}
\usepackage{color}
\usepackage{amssymb}
\usepackage{mathrsfs}
\usepackage[all]{xy}
\newtheorem{thm}{Theorem}[section]
\newtheorem{cor}[thm]{Corollary}
\newtheorem{lem}[thm]{Lemma}
\newtheorem{prop}[thm]{Proposition}
\theoremstyle{definition}
\newtheorem{defn}[thm]{Definition}
\newtheorem{rem}[thm]{Remark}

\newtheorem{example}[thm]{Example}
\numberwithin{equation}{section}
\newcommand{\re}{\mathfrak{Re}\,}
\newcommand{\im}{\mathfrak{Im}\,}

\def\epsilon{\varepsilon}

\newcommand{\fbl}{\text{FBL}}
\newcommand{\fvl}{\text{FVL}}
\newcommand{\vertiii}[1]{{\left\vert\kern-0.25ex\left\vert\kern-0.25ex\left\vert #1 
		\right\vert\kern-0.25ex\right\vert\kern-0.25ex\right\vert}}

\makeatletter
\@namedef{subjclassname@2020}{%
  \textup{2020} Mathematics Subject Classification}
\makeatother

\title{Free complex Banach lattices}

\author[D. de Hevia]{David de Hevia}
\address{Instituto de Ciencias Matem\'aticas (CSIC-UAM-UC3M-UCM)\\
Consejo Superior de Investigaciones Cient\'ificas\\
C/ Nicol\'as Cabrera, 13--15, Campus de Cantoblanco UAM\\
28049 Madrid, Spain.}
\email{david.dehevia@icmat.es}

\author[P. Tradacete]{Pedro Tradacete}
\address{Instituto de Ciencias Matem\'aticas (CSIC-UAM-UC3M-UCM)\\
Consejo Superior de Investigaciones Cient\'ificas\\
C/ Nicol\'as Cabrera, 13--15, Campus de Cantoblanco UAM\\
28049 Madrid, Spain.}
\email{pedro.tradacete@icmat.es}

 % Banach lattices, Operators in complex function spaces, Spectral sets of linear operators


\begin{document}

\setcounter{tocdepth}{1}
\date{\today}

\begin{abstract}
The construction of the free Banach lattice generated by a real Banach space is extended to the complex setting. It is shown that for every complex Banach space $E$ there is a complex Banach lattice $\fbl_{\mathbb C}[E]$ containing a linear isometric copy of $E$ and satisfying the following universal property: for every complex Banach lattice $X_{\mathbb C}$, every operator $T:E\rightarrow X_{\mathbb C}$ admits a unique lattice homomorphic extension $\hat{T}:\fbl_{\mathbb C}[E]\rightarrow X_{\mathbb C}$ with $\|\hat{T}\|=\|T\|$. The free complex Banach lattice $\fbl_{\mathbb C}[E]$ is shown to have analogous properties to those of its real counterpart. However, examples of non-isomorphic complex Banach spaces $E$ and $F$ can be given so that $\fbl_{\mathbb C}[E]$ and $\fbl_{\mathbb C}[F]$ are lattice isometric. The spectral theory of induced lattice homomorphisms on $\fbl_{\mathbb C}[E]$ is also explored.
\end{abstract}

\maketitle

\section{Introduction}

Free Banach lattices are a recent discovery in the theory of lattice structures and a new specimen of free object which has already proven its versatility to tackle certain problems in the interplay between Banach spaces and Banach lattices. Although free vector lattices had been known since the 1960s \cite{Baker,Bleier}, it was only in \cite{dePW} that B. de Pagter and A. W. Wickstead showed the existence of the free Banach lattice over a set. Motivated by the latter, A. Avil\'es, J. Rodr\'iguez and the second author introduced the free Banach lattice generated by a (real) Banach space. Given a Banach space $E$, the free Banach lattice generated by $E$ is a Banach lattice, denoted $\fbl[E]$, with the following properties:
\begin{enumerate}
\item there exists a linear isometric embedding $\delta_E: E\rightarrow \fbl[E]$,
\item for every Banach lattice $X$ and every operator $T:E\rightarrow X$ there exists a unique lattice homomorphism $\hat T:\fbl[E]\rightarrow X$ such that $\|\hat T\|=\|T\|$ and $\hat T\circ \delta_E=T$, i.e. the following diagram commutes:
$$
	\xymatrix{\fbl[E]\ar@{-->}^{\hat{T}}[drr]&& \\
	E\ar_{\delta_E}[u]\ar[rr]^T&&X}
$$
\end{enumerate}
The proof of existence of $\fbl[E]$ for every Banach space $E$, together with an explicit construction (to be recalled below) was given in \cite{ART}.

The assignment $E\mapsto \fbl[E]$ can be considered as a canonical functor between the category of Banach spaces (with bounded linear operators) and the category of Banach lattices (with lattice homomorphisms). Thus, it has been so far quite useful for addressing several questions on the interplay between these categories, such as the following ones: this construction has been used to provide examples of Banach lattices which are weakly-compactly generated as lattices but not as spaces, solving a question posed by J. Diestel \cite[Section 5]{ART}; it has been used to construct push-outs in the category of Banach lattices \cite{AT}, or to provide the first examples of lattice homomorphisms which do not attain their norm \cite{DMRR}; and, it has been used to show the existence of subspaces of Banach lattices without bibasic sequences \cite{OTTT}, solving a question from \cite{TT}.

Other recent relevant developments related to free Banach lattices include, among others, constructions of free Banach lattices over a lattice \cite{AMRR2,AR1}, study of chain conditions in $\fbl[E]$ \cite{APR}, existence of free Banach lattices with prescribed convexity conditions \cite{JLTTT}, relations with projectivity \cite{AMRT, AMR, AMR2}...

It should be noted that in all the above constructions these have only been considered for spaces (and lattices) over the real scalars. The purpose of this note is to provide an analogous construction in the complex setting, and explore similarities and differences with the setting of real scalars.

Since a complex Banach lattice is necessarily an appropriate complexification of a real Banach lattice, it is not unwise to conjecture that the free complex Banach lattice should be the complexification of some (real) free Banach lattice. In fact, given a complex Banach space one is naturally lead to consider $\fbl[E_\mathbb R]$, where $E_\mathbb R$ is just the space $E$ considered as a Banach space over $\mathbb R$, together with the canonical embedding $\delta_E:E\rightarrow \fbl[E_\mathbb R]_{\mathbb C}$ given as the complexification of $\delta_{E_\mathbb R}:E_\mathbb R\rightarrow \fbl[E_\mathbb R]$. However, this direct approach lacks a desirable stability concerning isometric properties and moreover fails the universal property when extending operators from $E$ to a complex Banach lattice which are not complexified operators. To fix this, we have to introduce an appropriate norm on $\fbl[E_\mathbb R]$, before complexifying, to get $\fbl_{\mathbb C}[E]$, as well as carefully define the canonical embedding $\delta_E:E\rightarrow \fbl_{\mathbb C}[E]$.

It is clear from the definition that linearly isomorphic (respectively isometric) real Banach spaces generate lattice isomorphic (respectively, isometric) free Banach lattices. A fundamental open question is whether there can exist non-isomorphic real Banach spaces $E$, $F$ so that $\fbl[E]$ and $\fbl[F]$ are lattice isomorphic. We will show that in the complex case there exist non-isomorphic complex Banach spaces whose corresponding free complex Banach lattices are isomorphic. In addition, it has been shown in \cite{OTTT} that under certain smoothness assumptions, $\fbl[E]$ and $\fbl[F]$ are lattice isometric precisely when $E$ and $F$ are linearly isometric. We will see here that there is an analogous statement for complex scalars.

The paper is organized as follows: after recalling some preliminaries on $\fbl[E]$ and complex Banach lattices in Section \ref{s:preliminaries}, we provide the explicit construction of $\fbl_{\mathbb C}[E]$ in Section \ref{s:fblC}, partly based on that of the free Banach lattice over the reals. To be more specific, for a complex Banach space $E$, the free complex Banach lattice $\fbl_{\mathbb C}[E]$ will be the complexification of an appropriate renorming of $\fbl[E_{\mathbb R}]$, where $E_{\mathbb R}$ is the space $E$ considered as a real Banach space. In Section \ref{s:conjugates}, we investigate the relation between $\fbl_{\mathbb{C}}$ and complex conjugates. In particular, we first show (Proposition \ref{p:complex conjugate}) that $\fbl_{\mathbb C}[E]$ is lattice isometric to $\fbl_{\mathbb C}[\overline{E}]$, where $\overline{E}$ denotes the complex conjugate of a complex Banach space $E$. This fact, in combination with the known examples of Banach spaces non-isomorphic to its complex conjugates \cite{Bourgain, Kalton:95} provide the first instances of non-isomorphic complex Banach spaces whose free Banach lattices are lattice isometric. A partial converse to Proposition \ref{p:complex conjugate}, under the assumption of smoothness of the duals is given in Proposition \ref{p:smoothdual}. In Section \ref{s:fvl}, we take a step back in order to consider free complex \emph{vector lattices} and show that a similar approach as in \cite{Troitsky} can be used to give an alternative proof of existence of $\fbl_{\mathbb C}$. Finally, in Section \ref{s:spectra}, given an endomorphism on a Banach space $T:E\to E$ we consider the induced lattice homomorphism $\overline{T}:\fbl_{\mathbb C}[E]\rightarrow \fbl_{\mathbb C}[E]$, and study the stability of the distinguished parts in the spectrum of $T$ with respect to those of $\overline{T}$. 

\section{Preliminaries}\label{s:preliminaries}

Let us start recalling that given a (real) Banach space $E$, \textit{the free Banach lattice generated by $E$}, denoted by $\fbl[E]$, can actually be constructed as follows: consider first the following expression on the set of functions $f:E^*\rightarrow \mathbb R$,
\begin{equation}\label{eq:free norm real}
    \|f\|_{\fbl[E]}=\sup\left\{\sum_{i=1}^n |f(x_i^*)|\,:\,n\in\mathbb{N}, \,\sup_{x\in B_E}\sum_{i=1}^n |x_i^*(x)|\leq 1\right\}.
\end{equation}
Note that $H_1[E]$, the space of positively homogeneous functions $f:E^*\rightarrow \mathbb R$ which satisfy $\|f\|_{\fbl[E]}<\infty$, with the pointwise ordering and the above norm is a Banach lattice. For each $x\in E$, let $\delta_x:E^*\rightarrow \mathbb R$ denote the evaluation function $\delta_x(x^*)=x^*(x)$. Then $\fbl[E]$ can be identified with the closed sublattice of $H_1[E]$ generated by the set $\{\delta_x:x\in E\}$ \cite[Theorem 2.5]{ART}.

Every operator between Banach spaces $T:E\rightarrow F$ admits a unique lattice homomorphism $\overline{T}:\fbl[E]\rightarrow \fbl[F]$ which extends $T$, in the sense that $\overline{T}\delta_E=\delta_F T$. Since the composition of lattice homomorphism is a lattice homomorphism, we have that $\overline{S\circ T}=\overline{S}\circ\overline{T}$ whenever the operators $S$ and $T$ (and their composition) are well defined. As a consequence, linearly isomorphic Banach spaces generate lattice isomorphic free Banach lattices. However, it is not known whether the converse of this statement holds.

Despite of this, in \cite{OTTT} several Banach space properties have been identified for a Banach space $E$ to be equivalent to certain lattice properties of $\fbl[E]$. These include for instance, finite dimensionality, separability or containing complemented copies of $\ell_1$.

For convenience, let us fixed the terminology concerning complex Banach lattices. By a \textit{complex Banach lattice} we mean the complexification $X_{\mathbb C}=X\oplus iX$ of a real Banach lattice, $X$, endowed with the norm $\|x+iy\|_{X_\mathbb{C}}=\||x+iy|\|_X$, where $|\cdot|:X_\mathbb{C}\to X^+$ is the mapping given by
\begin{equation}\label{eq:modulus}
|x+iy|=\sup_{\theta\in [0,2\pi]}\{x\cos\theta +y\sin \theta\}, \qquad \text{for every } x+iy\in X_\mathbb{C},    
\end{equation}
which is called \textit{the modulus} function. We refer to \cite[Section 3.2]{AA-book}, \cite[Section 2.11]{Schaefer-book} for a proof of the fact that the modulus is well-defined.

A complex subspace $Y$ of a complex Banach lattice $X_\mathbb{C}$ is said to be a \textit{complex sublattice} if $|x+iy|\in Y$ whenever $x+iy\in Y$. Equivalently, we can define a complex sublattice $Y$ of $X_\mathbb{C}$ as the complexification of a real sublattice $Z$ of $X$ \cite[Lemma 1.3]{Raynaud}. Similarly, a complex subspace $I$ of a complex Banach lattice $X_\mathbb{C}$ is said to be a \textit{complex ideal} if whenever $|z|\leq |w|$ and $w\in I$ imply that $z\in I$. Equivalently, a complex ideal $I$ of $X_\mathbb{C}$ can be defined as the complexification $I=I_0\oplus iI_0$ of a real ideal $I_0$ of $X$ \cite[Section 3.2, Ex. 7]{AA-book}.

Moreover, recall that for every $\mathbb{C}$-linear operator $T:X_\mathbb{C}\to Y_\mathbb{C}$ between two complex Banach lattices $X_\mathbb{C}$, $Y_\mathbb{C}$ there exists a unique pair of operators $T_1,T_2:X\to Y$ such that
$$T(x+iy)=(T_1+iT_2)(x+iy)=T_1x-T_2y+i(T_2x+T_1y),$$ 
for every $x+iy\in X_\mathbb{C}$ (see, for example, \cite[Section 1.1]{AA-book}).
If $T_2=0$, that is, if $T(x+iy)=S x+iS y$ for some operator $S:X\to Y$, we say that $T$ is a \textit{real operator} and that $T$ is the \textit{complexification} of $S$, written $T=S_\mathbb{C}$. On some occasions, for convenience, we will use the same symbol to represent an operator $T:X\to Y$ and its complexification, that is, $T(x+iy)=Tx+iTy$, for $x+iy\in X_\mathbb{C}$.

A real operator $T_\mathbb{C}:X_\mathbb{C}\to Y_\mathbb{C}$ is said to be \textit{positive} (resp. a \textit{lattice homomorphism}) if $T:X\to Y$ is positive (resp. a \textit{lattice homomorphism}). Complex lattice homomorphisms may be defined also as those $\mathbb{C}$-linear operators which preserve the modulus, that is, $T|z|=|Tz|$ for every $z\in X_\mathbb{C}$.

For a complex Banach space $E$, its dual space $E^*$ consists of all bounded $\mathbb C$-linear maps between $E$ and $\mathbb C$. Any complex Banach space $E$ can be seen as a real Banach space $E_\mathbb{R}$ if we restrict the scalar multiplication (of $E$) to the reals. Moreover, given $z^*\in E^*$ we can consider its real part $\re z^*$, which is an element of $(E_\mathbb R)^*$ and $\|z^*
\|_{E^*}=\|\re z^*\|_{(E_\mathbb R)^*}$. Conversely, given $x^*\in (E_\mathbb R)^*$, we can define
$$z^*(z)=x^*(z)-ix^*(iz), \qquad \text{ for every } z\in E,$$
and $\|z^*\|_{E^*}=\|x^*\|_{(E_\mathbb R)^*}$. The previous comments show that $(E^*)_\mathbb{R}$ and $(E_\mathbb R)^*$ are linearly isometric (cf. \cite[Theorem 1.9]{AA-book}).

\section{Construction of $\fbl_{\mathbb C}[E]$}\label{s:fblC}

\begin{defn}
Given a complex Banach space $E$, the \textbf{free complex Banach lattice} generated by $E$ is a \textsl{complex} Banach lattice $\fbl_{\mathbb C}[E]$ together with a $\mathbb C$-linear isometric embedding $\delta_E:E\rightarrow \fbl_{\mathbb C}[E]$ such that for every complex Banach lattice $X_\mathbb{C}$ and every $\mathbb C$-linear operator $T:E\rightarrow X_\mathbb{C}$, there is a unique lattice homomorphism $\hat T:\fbl_{\mathbb C}[E]\rightarrow X_\mathbb{C}$ such that $\hat T\circ \delta_E=T$. Moreover, $\|\hat T\|=\|T\|$.
\end{defn}

Observe that if this object exists for a complex Banach space $E$, then it is essentially unique in the sense that if there exists any other complex Banach lattice $L_\mathbb{C}$ with the previous property, we have that $L_\mathbb{C}$ is lattice isometric to $\fbl_\mathbb{C}[E]$

In this section we shall prove the existence of the free complex Banach lattice generated by a complex Banach space, providing an explicit description of this object.  Given a complex Banach space $E$, we can consider the real Banach space $E_\mathbb R$. We can equip the vector lattice $\fbl[E_\mathbb R]$ with the following norm:
$$
 \|f\|_{\fbl_{\mathbb C}[E]}=\sup\left\{\sum_{j=1}^m|f(\re z_j^*)|\,:\,m\in\mathbb N, (z_j^*)_{j=1}^m\subset E^*,\,\sup_{z\in B_E}\sum_{j=1}^m|z_j^*(z)|\leq1\right\}.   
$$
It should be noted that $\|\cdot\|_{\fbl_{\mathbb C}[E]}$ is a lattice norm in $\fbl[E_\mathbb{R}]$ (with the pointwise ordering) and is equivalent to the (real) free Banach lattice norm recalled in (\ref{eq:free norm real}):
\begin{equation}\label{eq:relation-norms}
\frac{1}{2}\| f \|_{\fbl[E_\mathbb R]}\leq \|f\|_{\fbl_{\mathbb C}[E]}\leq  \|f\|_{\fbl[E_\mathbb R]}, \qquad f\in \fbl[E_\mathbb R].    
\end{equation}

We define $\fbl_\mathbb{C}[E]$ as the complexification of the real Banach lattice $\fbl[E_\mathbb{R}]$ endowed with the complex Banach lattice norm $\||\cdot|\|_{\fbl_\mathbb{C}[E]}$. Observe that if $f=f_1+if_2\in \fbl_\mathbb{C}[E]=\fbl[E_\mathbb R]\oplus i \fbl[E_\mathbb R]$, the modulus of $f$ is given by
$$
|f|(x^*)=\sqrt{f_1(x^*)^2+f_2(x^*)^2}, \qquad \text{for every } x^*\in (E_\mathbb{R})^*,
$$
and, thus
\begin{equation}\label{eq:norm complex free}
 \||f|\|_{\fbl_{\mathbb C}[E]}=\sup\left\{\sum_{j=1}^m|f(\re z_j^*)|\,:\, (z_j^*)_{j=1}^m\subset E^*,\,\sup_{z\in B_E}\sum_{j=1}^m|z_j^*(z)|\leq1\right\}.       
\end{equation}
 For simplicity, henceforth, we shall take the above expression as the definition of  $\|\cdot\|_{\fbl_\mathbb{C}[E]}$, that is,  the norm $\| |\cdot|\|_{\fbl_\mathbb{C}[E]}$ will be represented by $\|\cdot\|_{\fbl_\mathbb{C}[E]}$.

Let $\delta_E:E\rightarrow \fbl_{\mathbb C}[E]$ be given by 
\begin{equation}\label{eq:deltaE}
\delta_E(z)=\delta_{E_\mathbb R}(z)-i\delta_{E_\mathbb R}(iz), \qquad z\in E.
\end{equation}

Observe that $\delta_E$ is a $\mathbb{C}$-linear map. Indeed, $\delta_E$ is $\mathbb{R}$-linear, as $\delta_{E_\mathbb{R}}$ has this property, and $\delta_E(iz)=\delta_{E_\mathbb R}(iz)+i\delta_{E_\mathbb R}(z)=i\delta_E(z)$. Moreover, this mapping is norm-preserving.

\begin{lem}\label{l:delta}
The map $\delta_E$ is a $\mathbb C$-linear isometric embedding.
\end{lem}

\begin{proof}
Let us note that for every $z\in E$ and for every $z^*\in E^*$ we have that
\begin{eqnarray*}
\delta_E(z)(\re z^*)&=&\delta_{E_\mathbb R}(z)(\re z^*)-i\delta_{E_\mathbb R}(iz)(\re z^*)=\re z^*(z)-i\re z^*(iz) \\
&=&\re z^*(z)+i\im z^*(z)=z^*(z).
\end{eqnarray*}
Using the above identity, it is straightforward to check that $\|\delta_E(z)\|_{\fbl_\mathbb{C}[E]}=\|z\|$ for all $z\in E$ in view of the definition (\ref{eq:norm complex free}) of $\|\cdot\|_{\fbl[E]_\mathbb{C}}$.
\end{proof}

\begin{thm}
The complex Banach lattice $\fbl_\mathbb{C}[E]=\fbl[E_\mathbb R]\oplus i \fbl[E_\mathbb R]$ with the norm $\|\cdot\|_{\fbl_\mathbb C[E]}$, together with the map $\delta_E$ given above, form the free complex Banach lattice generated by $E$.
\end{thm}

\begin{proof}
By Lemma \ref{l:delta}, $\delta_E$ is a $\mathbb C$-linear isometric embedding.

Given a complex Banach lattice $X_\mathbb{C}=X\oplus iX$, where $X$ is a (real) Banach lattice, we can consider the projection onto the real part $\re:X_\mathbb{C}\rightarrow X$ given by $\re(x+iy)=x$ for $x\in X$. This defines a $\mathbb{R}$-linear projection. For a $\mathbb C$-linear operator $T:E\rightarrow X_\mathbb{C}$, let $S:\fbl[E_\mathbb R]\rightarrow X$ denote the unique lattice homomorphism such that $S\circ  \delta_{E_\mathbb R}=\re \circ T$. Let $\hat T:\fbl[E_\mathbb R]_{\mathbb C}\rightarrow X_\mathbb{C}$ be the complexification of the operator $S$, that is, $\hat T(f+ig)=Sf+iSg$. 
Thus, $\hat T$ is a complex lattice homomorphism. 

Moreover, for every $z\in E$, using the fact that $T$ is $\mathbb C$-linear, we have that
\begin{align*}
\hat T\delta_E(z)&=\hat T(\delta_{E_\mathbb R}(z)-i\delta_{E_\mathbb R}(iz))=S\delta_{E_\mathbb R}(z)-iS\delta_{E_\mathbb R}(iz)\\
&=\re T (z)-i\re T(iz)=\re T (z)-i\re iT(z)=T(z),
\end{align*}
so $\hat T {\delta_E}=T$. 

Now, let us see the uniqueness of the extension $\hat{T}$. Let $\hat{T}_1:\fbl[E_\mathbb{R}]_\mathbb{C}\to X_\mathbb{C}$ be another complex lattice homomorphism such that $\hat{T}_1\delta_E=T$. As a complex lattice homomorphism, $\hat{T}_1$ satisfies $\hat{T}_1(f+ig)=S_1f+iS_1g$ for every $f,g\in \fbl[E_\mathbb{R}]$ for some lattice homomorphism $S_1$. By the definition of $\delta_E$, it follows that $S_1\delta_{E_\mathbb{R}}=\re T=S\delta_{E_\mathbb{R}}$. Since $\fbl[E_{\mathbb{R}}]=\overline{\text{lat}}\{\delta_{E_\mathbb{R}}(x)\::\:x\in E\}$, we conclude that $S$ and $S_1$ agree on $\fbl[E_{\mathbb{R}}]$ and, consequently, $\hat{T}=\hat{T}_1$.

We claim that $\|\hat T\|=\|T\|$. First, it should be noted that since $\hat{T}\delta_E=T$ and $\delta_E$ is an isometry, we have that $\|T\|\leq \|\hat{T}\|$. It remains to show that
$$
\||\hat{T}(f+ig)|\|_{X}\leq\|T\|\|f+ig\|_{\fbl_\mathbb{C}[E]},
$$ 
for all $f,g\in \fbl[E_\mathbb{R}]$. Since $\hat{T}$ is a lattice homomorphism, the preceding inequality is equivalent to
\begin{equation}\label{eq:main-inequality}
\|Sf\|_X\leq \|T\| \|f\|_{\fbl_\mathbb{C}[E]}, \qquad \text{ for every } f\in (\fbl[E_\mathbb{R}])^+.    
\end{equation}
Moreover, by density of $\text{lat}\{\delta_{E_\mathbb{R}}(x)\::\:x\in E\}$ in $\fbl[E_\mathbb{R}]$, it suffices to check the above identity when $f$ has the form (see, for instance \cite[Section 4.1, Ex. 8]{AB})  $$f=\bigvee_{i=1}^p \delta_{E_\mathbb{R}}(x_i)-\bigvee_{j=1}^q \delta_{E_\mathbb{R}}(y_j).$$

Fix a positive element
$$
f=\bigvee_{i=1}^p \delta_{E_\mathbb{R}}(x_i)-\bigvee_{j=1}^q \delta_{E_\mathbb{R}}(y_j).
$$
Since $Sf\geq 0$, equation (\ref{eq:main-inequality}) is equivalent to $y^*(Sf)\leq \|T\|\|f\|_{\fbl_\mathbb{C}[E]}$ for every $y^*\in (B_{X^*})^+$ (see \cite[Proposition 1.3.5]{Meyer-Nieberg}). Take an arbitrary decomposition $y^*=\sum_{k=1}^p y_k^*$, where $y_1^*,\ldots,y_p^*\in (X^*)^+$. For every $k\in \{1,\ldots,p\}$, define 
$$
x_k^*=\|T\|^{-1}(\re T)^*(y_k^*)\in (E_\mathbb{R})^*
$$ 
(note that $\re T$ is a $\mathbb{R}$-linear operator from $E_\mathbb{R}$ to $X$). Hence, if we put 
$$
z^*_k(z)=x_k^*(z)-ix_k^*(iz),\quad \text{ for } z\in E,
$$
then $z_k^*$ defines a $\mathbb{C}$-linear functional on $E$ with real part $\re z_k^*=x_k^*$  for $k=1,\ldots,p$. 

For each $z\in B_E$ and for each $k\in\{1,\ldots, p\}$ let $\theta_{z,k}$ be a real number such that 
$$
|z_k^*(z)|e^{i\theta_{z,k}}=z_k^*(z).
$$
Using in step $(*)$ the complex homogeneity of $T$, we have that
\begin{eqnarray*}
\sup_{z\in B_E} \sum_{k=1}^p  |z_k^*(z)|&=&\sup_{z\in B_E} \sum_{k=1}^p e^{-i\theta_{z,k}}z_k^*(z)=\sup_{z\in B_E} \sum_{k=1}^p z_k^*\left(e^{-i\theta_{z,k}}z\right) \\ 
&=& \sup_{z\in B_E} \sum_{k=1}^p x_k^*\left(e^{-i\theta_{z,k}}z\right)=\sup_{z\in B_E} \sum_{k=1}^p \frac{1}{\|T\|}y_k^*\left(\re T\left( e^{-i\theta_{z,k}}z\right)\right) \\ 
&\leq & \sup_{z\in B_E} \sum_{k=1}^p \frac{1}{\|T\|}y_k^*\left(\bigl|T\left(e^{-i\theta_{z,k}}z \right) \bigr| \right)\overset{(*)}{=} \sup_{z\in B_E} \sum_{k=1}^p y_k^*\left(\frac{1}{\|T\|}|T(z)| \right) \\ 
&=& \sup_{z\in B_E} y^*\left(\frac{1}{\|T\|}|T(z)| \right)\leq \|y^*\|\leq 1. 
\end{eqnarray*} 

By the definition of the norm $\|\cdot \|_{\fbl_\mathbb{C}[E]}$ it follows that
\begin{eqnarray*}
\|f\|_{\fbl_\mathbb{C}[E]}&\geq& \sum_{k=1}^p f(\re z_k^*)=\sum_{k=1}^p f(x_k^*)=\sum_{k=1}^p \left(\bigvee_{i=1}^p \delta_{E_\mathbb{R}}(x_i)(x_k^*)-\bigvee_{j=1}^q \delta_{E_\mathbb{R}}(y_j)(x_k^*) \right) \\
&=&\frac{1}{\|T\|} \sum_{k=1}^p \left(\bigvee_{i=1}^p y_k^*(\re Tx_i)-\bigvee_{j=1}^q y_k^*(\re Ty_j)  \right) \\
&\geq & \frac{1}{\|T\|} \sum_{k=1}^p \left(y_k^*(\re Tx_k)-y_k^*\left(\bigvee_{j=1}^q \re Ty_j \right) \right) \\
&=& \frac{1}{\|T\|}\left(\sum_{k=1}^p y_k^*(\re Tx_k)-y^*\left(\bigvee_{j=1}^q \re Ty_j \right)\right).
\end{eqnarray*}
If we take the supremum over all decompositions of $y^*$ into $p$ positive elements of $X^*$, bearing in mind the Riesz-Kantorovich formulas \cite[Theorem 1.21]{AB}, we obtain from the previous expression that
$$\|T\|\|f\|_{\fbl_\mathbb{C}[E]}\geq y^*\left(\bigvee_{k=1}^p \re Tx_k- \bigvee_{j=1}^q \re Ty_j \right)=y^*(Sf).$$
Remember that $S$ is a lattice homomorphism such that $S\delta_{E_\mathbb{R}}=\re T$. This concludes the proof.
\end{proof}

It should be recalled that the lattice homomorphisms of $\fbl[E]^*$ are precisely the evaluations functionals $\varphi_{x^*}(f)=f(x^*)$, $f\in \fbl[E]$, with $x^*\in E^*$ \cite[Corollary 2.7]{ART}. Now, we establish an analogous result in the complex case.

\begin{cor}
$\varphi\in \fbl_\mathbb C[E]^*$ is a lattice homomorphism if and only if there is $z^*\in E^*$ such that $\varphi(f+ig)=f(\re z^*)+ig(\re z^*)$ for all $f+ig\in \fbl_\mathbb C[E]$.
\end{cor}

\begin{proof} 
It is clear that the evaluation functional 
$$
\varphi_{\re z^*}(f+ig)=f(\re z^*)+ig(\re z^*)
$$
is a lattice homomorphism for any $z^*\in E^*$. 

Conversely, let $\varphi$ be a lattice homomorphism in $\fbl_\mathbb{C}[E]^*$. We define $z^*=\varphi\circ \delta_E\in E^*$. Then, $\varphi_{\re z^*}$ is a lattice homomorphism in $\fbl_\mathbb{C}[E]^*$ such that 
$$
\varphi_{\re z^*}\circ\delta_E=z^*=\varphi\circ\delta_E.
$$
By the uniqueness of the universal property of $\fbl_\mathbb{C}[E]$ we conclude that $\varphi=\varphi_{\re z^*}$.
\end{proof}

\begin{rem}
In a similar spirit as in \cite{JLTTT}, for $p\geq 1$ one can define $\fbl^p_{\mathbb C}[E]$, the free $p$-convex complex Banach lattice generated by a complex Banach space $E$. This can be done replacing the above norm with
$$
\|f\|_{\fbl^p_{\mathbb C}[E]}=\sup\left\{\left(\sum_{j=1}^m|f(\re z_j^*)|^p\right)^{\frac{1}{p}}\::\: (z_j^*)_{j=1}^m\subset E^*,\,\sup_{z\in B_E}\sum_{j=1}^m|z_j^*(z)|^p\leq1\right\}.
$$
For $p=\infty$, similarly to the real setting, it can be shown that $\fbl_\mathbb{C}^{(\infty)}[E]$ coincides with the sublattice generated by $\{\delta_E(z)\}_{z\in E}$ in $\mathcal{C}(B_{E^*})_\mathbb{C}$, which is precisely $\mathcal{C}_{ph}(B_{E^*})_{\mathbb C}$ (the space of all positively homogeneous $\text{weak}^*$-continuous functions from $B_{E^*}$ to $\mathbb C$).
\end{rem}

\begin{rem}
Let $E$ be a real Banach space. Since any function in $\fbl[E]$ is $w^*$-continuous on $B_{E^*}$ (see \cite[Lemma 4.10]{ART}), it follows that $\delta_E(E)$ is precisely the subset of functions of $\fbl[E]$ which are linear. The last statement remains true if we replace the word ``linear" by ``additive", given that functions of $\fbl[E]$ are positively homogeneous. We may formulate a similar result to the previous one in the complex case. Let $E$ be a complex Banach space. A function $f+ig\in \fbl_\mathbb{C}[E]$ belongs to $\delta_E(E)$ if, and only if, the map $z^*\in E^*\mapsto f(\re z^*)+ig(\re z^*)$ is $\mathbb{C}$-linear. To be a $\delta_E(z)$, for some $z\in E$, is also equivalent to the fact of $f+ig$ being additive and satisfying 
$$
(f+ig)(L_i x^*)=i(f+ig)(x^*), \quad \text{ for every } x^*\in (E_\mathbb{R})^*,
$$ 
where $L_i:(E_\mathbb{R})^*\to (E_\mathbb{R})^*$ is defined by $L_ix^*(x)=x^*(ix)$ for all $x^*\in (E_\mathbb{R})^*$ and all $x\in E$.
\end{rem}

\section{Complex conjugates and $\fbl_\mathbb C[E]$}\label{s:conjugates}
As was mentioned in the Introduction, in the real setting it is an open question whether there can exist non-isomorphic Banach spaces such that their corresponding free Banach lattices are lattice isomorphic. In this section we will analyze this problem in the complex case. 

If $E$ is a complex Banach space, its \textsl{complex conjugate} $\overline E$ is defined as the space $E$ with the scalar multiplication $\alpha\odot x=\overline\alpha x$. If $E=F_\mathbb C$ for some real Banach space $F$, the map $T:E\rightarrow \overline E$ given by $T(x+iy)=y+ix$ is easily seen to be a $\mathbb C$-linear isomorphism. The first example of a complex Banach space which is not $\mathbb C$-isomorphic to its complex conjugate is due to Bourgain \cite{Bourgain}. An elementary explicit example was later given by Kalton \cite{Kalton:95}.

\begin{prop}\label{p:complex conjugate}
For every complex Banach space $E$, $\fbl_{\mathbb C}[E]$ is lattice isometric to $\fbl_{\mathbb C}[\overline E]$.
\end{prop}

\begin{proof} First, notice that $E_\mathbb{R}=\overline{E}_\mathbb{R}$. Hence, with the construction of the free complex Banach lattice detailed in the preceding section, we obtain that $\fbl_\mathbb{C}[E]$ and $\fbl_\mathbb{C}[\overline{E}]$ are the same set. We are going to show that the identity $f\in \fbl_\mathbb{C}[E]\mapsto f\in \fbl_\mathbb{C}[\overline{E}]$ is norm-preserving. 

Let $z^*$ be an element of $E^*$. If we define 
$$
\psi_{z^*}(z)=\re z^*(z)-i\im z^*(z),
$$ 
for every $z\in \overline{E}$, we obtain a bounded functional on $\overline{E}$. It is clear that $\psi_{z^*}$ is $\mathbb{R}$-linear, so it suffices to check that $\psi_{z^*}(i\odot z)=i\psi_{z^*}(z)$. Indeed, for every $z\in \overline{E}$ we have
\begin{eqnarray*}
\psi_{z^*}(i\odot z) &=& -\psi_{z^*}(iz)= -\re z^*(iz)+i\im z^*(iz)=\im z^*(z)+i\re z^*(z) \\
 &=& i(\re z^*(z)-i\im z^*(z)) =i\psi_{z^*}(z).
\end{eqnarray*}
In addition, note that the real parts and moduli of $z^*$ and $\psi_{z^*}$ agree on the set $E$. Conversely, if we begin with an element $w^*\in \overline{E}^*$, we can construct similarly a bounded functional on $E^*$ which has the same real part and modulus as $w^*$ at each point of $E$. It follows that
\begin{eqnarray*}
 \|f\|_{\fbl_{\mathbb C}[E]}&=& \sup\left\{\underset{j=1}{\overset{m}\sum}|f(\re z_j^*)|\,:\, (z_j^*)_{j=1}^m\subset E^*,\,\sup_{z\in B_{E}}\underset{j=1}{\overset{m}\sum}|z_j^*(z)|\leq1\right\}  \\
  &=& \sup\left\{\underset{j=1}{\overset{m}\sum}|f(\re \psi_{z_j^*})|\,:\, (z_j^*)_{j=1}^m\subset E^*,\,\sup_{z\in B_{\overline{E}}}\underset{j=1}{\overset{m}\sum}|\psi_{z_j^*}(z)|\leq1\right\} \\
 &=&\sup\left\{\underset{j=1}{\overset{m}\sum}|f(\re w_j^*)|\,:\, (w_j^*)_{j=1}^m\subset \overline{E}^*,\,\sup_{z\in B_{\overline{E}}}\underset{j=1}{\overset{m}\sum}|w_j^*(z)|\leq1\right\}  \\
&=& \|f\|_{\fbl_{\mathbb C}[\overline{E}]}.
\end{eqnarray*}

\end{proof}

\begin{cor}\label{c:isomorphic}
There exist non-isomorphic complex Banach spaces $E$ and $F$ such that $\fbl_{\mathbb C}[E]$ and $\fbl_{\mathbb C}[F]$ are lattice isometric.
\end{cor}

\begin{proof}
Take any complex Banach space $E$ non-isomorphic to its complex conjugate, then apply Proposition \ref{p:complex conjugate} with $F=\overline E$.
\end{proof}

\begin{rem}
Corollary \ref{c:isomorphic} is not known to hold for real spaces.
\end{rem}

The next proposition exhibits that the lattice homomorphisms between complex free Banach lattices are \textit{composition operators}. We shall omit the proof of this fact since it can be readily adjusted from its real version \cite{LaTra,OTTT}.

\begin{prop}\label{prop:44}
Given two complex Banach spaces $E$ and $F$ and a (complex) lattice homomorphism $T:\fbl_\mathbb{C}[F]\to \fbl_{\mathbb{C}}[E]$, we define a map $\Phi_T:E^*\to F^*$ given by
$$\Phi_T z^*(w)=(T\delta_F(w))(\re z^*), \quad \text{for every }\; z^*\in E^*, \;\; w\in F.$$
The above map satisfies the following properties:

\begin{enumerate}
    \item  For every $f+ig \in \fbl_\mathbb{C}[E]$ we have that $T(f+ig)=(f+ig)\circ \Phi_T^{\re}$, where $\Phi_T^{\re}:(E_\mathbb{R})^*\to (F_\mathbb{R})^*$ is defined by $\Phi_T^{\re}(\re z^*)=\re (\Phi_T z^*)$.
    \item $\Phi_T$ is positively homogeneous.
    \item $\Phi_T$ is $w^*$ to $w^*$ continuous on bounded sets.
    \item For every $(z_j^*)_{j=1}^m\subset E^*$ we have
    \begin{equation}\label{eq:41}
       \sup_{w\in B_F} \sum_{j=1}^m |\Phi_T z_j^*(w)| \leq \|T\| \sup_{z\in B_E}\sum_{j=1}^m | z_j^*(z)|. 
    \end{equation}
\end{enumerate}
\end{prop}

\begin{rem}
Occasionally, it may be helpful to keep in mind the following identity for the $(1,weak)$-norm of a finite sequence $(z_j^*)_{j=1}^m\subset E^*$: $$ \sup_{z\in B_E}\sum_{j=1}^m | z_j^*(z)|=\sup\left\{\bigl\|\sum_{j=1}^m \varepsilon_j z^*_j \bigr\|\::\:|\varepsilon_j|=1 \text{ for } j=1,\ldots,m\right\}.$$
\end{rem}

\begin{rem}
If $T$ is a lattice isomorphism, then $\Phi_T$ is bijective and $\Phi_T^{-1}=\Phi_{T^{-1}}$. Therefore, if $T$ is also an isometry, then we deduce from the inequality (\ref{eq:41}) that $\Phi_T$ (and, for the same reason, $\Phi_T^{-1}$) preserves $(1,weak)$-norms of finite sequences.
\end{rem}

Recall that given a Banach space $Z$, a supporting functional at a point $z\in Z$ is an element $f_z\in Z^*$ such that $\|f_z\|=1$ and $f_z(z)=\|z\|$. Recall that $Z$ is said to be \textit{smooth} if for every element $z\in Z$, with $z\neq 0$, there exists a unique supporting functional $f_z$ at $z$. 

It was shown in \cite{OTTT} that if $E$ and $F$ are real Banach spaces with smooth duals, then every lattice isometry $T:\fbl[E]\rightarrow \fbl[F]$ is necessarily induced by a linear isometry between $E$ and $F$. Now, we are going to establish a version for complex scalars of this result in Proposition \ref{p:smoothdual}. It should be noted that this proposition provides a partial converse to Proposition \ref{p:complex conjugate}.

The following elementary observation will be crucial in the proof of the next lemma.

\begin{rem}\label{r:increasing}
For each pair $x,y$ of vectors in a Banach space $Z$, the function 
$$
t\in(0,+\infty)\mapsto \frac{\|x+ty\|-\|x\|}{t}
$$
is increasing. Indeed, given two positive real numbers $s<t$, by the convexity of the norm function $f(z)=\|z\|$, we have
\begin{eqnarray*}
\frac{f(x+sy)-f(x)}{s} &=& \frac{f\bigl(\frac{s}{t}(x+ty)+\bigl(1-\frac{s}{t} \bigr)x \bigr)-f(x)}{s} \\ 
&\leq&  \frac{\frac{s}{t}f(x+ty)+\bigl(1-\frac{s}{t} \bigr)f(x)-f(x)}{s} =\frac{f(x+ty)-f(x)}{t}.
\end{eqnarray*}
\end{rem}

It is well-known that a Banach space $Z$ is smooth at $z\in Z\backslash\{0\}$ if and only if the norm $\|\cdot\|$ (of $Z$) is \textit{G\^{a}teaux differentiable} at $z$, that is, if there exists $F(z)\in (Z_\mathbb{R})^*$ such that $$
F(z)(w)=\lim_{t\rightarrow0}\frac{\|z+tw\|-\|z\|}{t}, \quad \text{ for every } w\in Z, 
$$ 
where the above limit is assumed to be taken in $\mathbb{R}$ \cite[Corollary 1.5]{Deville}. Moreover, the \textit{G\^{a}teaux derivative} $F(z)$ is a supporting functional at $z$ (see, for instance, \cite[p. 2]{Deville}). Thus, in the case that $\|\cdot\|$ is smooth at $z\neq 0$, we have
$$\lim_{t\to 0} \frac{\|z+tw\|-\|z\|}{t}=\mathfrak{Re} f_{z}(w),$$
for every $w\in Z$.

\begin{lem}
Let $z,w$ be given in a smooth complex Banach space $Z$ with $z\neq 0$. Then, $$\lim_{t\to 0^+}\frac{\sup_{|\varepsilon|=1}\|z+\varepsilon tw\|-\|z\|}{t}=|f_z(w)|,$$
where $f_{z}\in Z^*$ stands for the unique supporting functional at $z$.
\end{lem}
\begin{proof}
Since $Z$ is smooth, we have that
$$\lim_{t\to 0} \frac{\|z+tw\|-\|z\|}{t}=\mathfrak{Re} f_{z}(w).$$
Thus, for every $\theta \in [0,2\pi]$ we have that
$$\lim_{t\to 0} \frac{\|z+te^{i\theta}w\|-\|z\|}{t}=\mathfrak{Re} f_{z}(e^{i\theta} w)=\mathfrak{Re} f_{z}(w)\cos \theta - \mathfrak{Im} f_{z}(w)\sin \theta,$$
so that there exists $\theta_0 \in [0, 2\pi ]$ such that $\lim_{t\to 0} \frac{\|z+te^{i\theta_0}w\|-\|z\|}{t}=|f_{z}(w)|$.

If $t>0$, then by Remark \ref{r:increasing} we have
$$
\frac{\sup_{|\varepsilon|=1}\|z+\varepsilon tw\|-\|z\|}{t} =\sup_{|\varepsilon|=1} \frac{\|z+\varepsilon tw\|-\|z\|}{t} \geq \frac{\|z+te^{i\theta_0}w\|-\|z\|}{t}\geq |f_{z}(w)|,
$$
and thus, 
$$
\liminf_{t\to 0^+} \frac{\sup_{|\varepsilon|=1}\|z+\varepsilon tw\|-\|z\|}{t}\geq |f_{z}(w)|.
$$

On the other hand, fix $\delta>0$. Given $|\varepsilon|=1$, there is $t_\varepsilon>0$ such that for all $t\in (0, t_\varepsilon)$ we have
$$\frac{\|z+\varepsilon tw\|-\|z\|}{t}< \lim_{t\to 0^+} \frac{\|z+\varepsilon tw\|-\|z\|}{t}+\delta\leq |f_{z}(w)|+\delta.$$
Thus, the family of sets
$$
U_t=\left\{\varepsilon\in \partial \mathbb{D}\::\: \frac{\|z+\varepsilon tw\|-\|z\|}{t}<|f_z(w)|+\delta\right\}, \quad \text{ for } t\in(0,+\infty),
$$
is an open cover of $\partial \mathbb{D}=\{\varepsilon\in\mathbb{C}\::\:|\varepsilon|=1\}$. By the compactness of $\partial \mathbb{D}$, there exist $t_1,\ldots, t_n$ such that $\partial \mathbb{D}\subset \cup_{i=1}^n U_{t_i}$. By Remark \ref{r:increasing}, we have that $\cup_{i=1}^n U_{t_i}=U_{t_0}$, where $t_0=\min_{1\leq i\leq n} t_i$.
Therefore, we get that
$$
\frac{\|z+\varepsilon tw\|-\|z\|}{t}<  |f_{z}(w)|+\delta, \quad \text{ for every } t\in (0,t_0) \text{ and every } |\varepsilon|=1.
$$
Taking supremum over all $|\varepsilon|=1$ in the above equation, and limit superior for $t\to 0^+$, we conclude that
$$\limsup_{t\to 0^+}\frac{\sup_{|\varepsilon|=1}\|z+\varepsilon tw\|-\|z\|}{t}\leq |f_{z}(w)|.$$
We have proven that
$$\lim_{t\to 0^+}\frac{\sup_{|\varepsilon|=1}\|z+\varepsilon tw\|-\|z\|}{t}=|f_z(w)|.$$
\end{proof}

%The following proposition provides a converse result to Proposition \ref{p:complex conjugate} under the assumption that we are dealing with complex Banach spaces with smooth dual. 

\begin{prop}\label{p:smoothdual}
Let $E$, $F$ be complex Banach spaces whose corresponding duals $E^*$, $F^*$ are assumed to be smooth. If $\fbl_\mathbb{C}[E]$ is lattice isometric to $\fbl_\mathbb{C}[F]$, then $E$ is isometric to $F$ or $\overline{F}$.
\end{prop}

\begin{proof}
Let $T:\fbl_\mathbb{C}[E]\to \fbl_\mathbb{C}[F]$ be a surjective lattice isometry. We define the following semi-inner product on $E^*$ (resp. on $F^*$): for $z^*,w^*\in E^*$ (resp. on $F^*$), 
$$
[w^*,z^*]=\left\{\begin{array}{cc}
    0 & \text{ if }z^*=0, \\
    f_{z^*}(w^*) & \text{ if }z^*\neq0,
\end{array}\right.
$$
where $f_{z^*}\in E^{**}$ (resp. $F^{**})$ is the unique supporting functional at $z^*$.

By Proposition \ref{prop:44}, $T$ is the composition operator associated to a certain surjective positively homogeneous map $\Phi_T:F^*\to E^*$, $\text{weak}^*$ to $\text{weak}^*$ continuous on bounded sets, which preserves $(1,weak)$-norms of tuples; $\Phi_T^{-1}$ also has all these properties. Then, by the previous lemma we have
\begin{eqnarray*}
\bigl|[w^*,z^*]\bigr|&=&|f_{z^*}(w^*)|=\lim_{t\to 0^+}\frac{\sup_{|\varepsilon|=1}\|z^*+\varepsilon tw^*\|-\|z^*\|}{t} \\
&=&\lim_{t\to 0^+}\frac{\sup_{|\varepsilon|=1}\|\Phi_Tz^*+\varepsilon t\Phi_Tw^*\|-\|\Phi_Tz^*\|}{t} \\
&=&|f_{\Phi_T z^*}(\Phi_T w^*)|=\bigl|[\Phi_T w^*,\Phi_T z^*]\bigr|,
\end{eqnarray*}
for every $z^*,w^*\in E^*$, with $z^*\neq 0$. By \cite{IT2020}, there exist a map $\sigma:F^*\to \mathbb{C}$ with $|\sigma(z^*)|=1$ for all $z^*\in F^*$ and a linear or conjugate linear surjective isometry $V:F^*\to E^*$ such that 
$$
\Phi_T z^*= \sigma(z^*) Vz^*.
$$
Therefore, we have $Vz^*=\overline{\sigma(z^*)}\Phi_T z^*$ for all $z^*\in F^*$. 

Note that $V$ is $w^*$-$w^*$ continuous on $B_{F^*}$. Indeed, since $V$ is linear, it suffices to check that $Vz_\alpha^*\overset{w^*}{\to} 0$ for any $\{z_\alpha^*\}\subset B_{F^*}$ $\text{weak}^*$-convergent to zero. This can easily be deduced from the fact that $\Phi_T$ is $\text{weak}^*$ to $\text{weak}^*$ continuous on bounded sets and $|\sigma(z^*)|=1$ for every $z^*\in F^*$.

Let us denote by $S:E^*\to \overline{E}^*$ the map defined by 
$$
Sz^*(z)=\overline{z^*(z)}
$$ 
for all $z^*\in E^*$, $z\in \overline{E}$. It should be noted that $S$ is a conjugate-linear surjective isometry which is also $\text{weak}^*$ to $\text{weak}^*$ continuous.

If $V$ is linear, given that it is weak$^*$ to weak$^*$ continuous on $B_{F^*}$, then $V$ is the adjoint operator of a surjective isometry between $E$ and $F$. If $V$ is conjugate-linear, we may take the composition $SV$ which must be the adjoint operator of a surjective isometry between $\overline{E}$ and $F$.
\end{proof}

\begin{rem}
The preceding proposition cannot be extended to the lattice isomorphic case. In \cite{Anisca}, R. Anisca built a family of cardinality continuum of uniformly convex Banach spaces (and thus, with uniformly smooth duals) which are mutually non-isomorphic as complex Banach spaces even though they are real isometric. Therefore, the members of this collection have essentially (up to lattice isomorphism) the same free complex Banach lattice.
\end{rem}

More generally, one might consider the notion of \textit{complex structures}. Recall that a real Banach space $E$ is said \textit{to admit a complex structure} if there exists an operator $U:E\to E$ such that $U^2=\text{-Id}$ (see \cite[pp. 4-5]{Singer}). In this situation, we can put the following scalar multiplication on $E$:
$$
(a+ib)\cdot x=ax+b\,Ux, \qquad \text{ for every } x\in E \:\text{ and every } a,b\in\mathbb{R}.
$$
Thus, $E$ becomes a complex Banach space, denoted by $E^U$, when renormed with
\begin{equation}\label{eq:renorming-complex structure}
\vertiii{x}=\sup_{\theta\in  [0,2\pi]} \|\cos \theta x+\sin \theta \,Ux \|, \qquad \text{ for every } x\in E.    
\end{equation}

Observe that from the construction of the free complex Banach lattice described in the preceding section we can infer that if $E$ and $F$ are two complex Banach spaces which are $\mathbb{R}$-linearly isomorphic, then their corresponding $\fbl_\mathbb{C}[E]$ and $\fbl_\mathbb{C}[F]$ are complex lattice isomorphic. Indeed, since $\fbl[E_\mathbb{R}]$ and $\fbl[F_\mathbb{R}]$ are lattice isomorphic, then $(\fbl[E_\mathbb{R}])_\mathbb{C}$ and $(\fbl[F_\mathbb{R}])_\mathbb{C}$ are complex lattice isomorphic. In addition, equation (\ref{eq:relation-norms}) shows that $\fbl_\mathbb{C}[E]$ (resp. $\fbl_\mathbb{C}[F]$) is complex lattice isomorphic to $(\fbl[E_\mathbb{R}])_\mathbb{C}$ (resp. $(\fbl[F_\mathbb{R}])_\mathbb{C}$). 

It should be noted that if $U$, $V$ are complex structures on $E$, then their associated complex Banach spaces $E^U$, $E^V$ are isomorphic as real Banach spaces, given that their norms are equivalent to the original one defined on $E$ (recall expression (\ref{eq:renorming-complex structure})). Nevertheless, $E^U$, $E^V$ do not have to be $\mathbb{C}$-linearly isomorphic. In fact, with this terminology, the spaces constructed in \cite{Bourgain,Kalton:95} have more than one complex structure, whereas spaces with exactly $n$ non-equivalent complex structures were given in \cite{Ferenczi}  (see also \cite{Anisca, Cuellar} for the cases of continuum many and infinite countably many non-equivalent complex structures respectively). As a result, these provide examples of non-isomorphic complex Banach spaces whose corresponding free complex Banach lattices are lattice isomorphic. 

\section{Free complex vector lattices}\label{s:fvl}

The purpose of this section is to provide an alternative proof of the existence of $\fbl_{\mathbb C}[E]$, similar to the one given in \cite{Troitsky}. Although this argument is conceptually simpler it has the drawback that it does not provide the explicit description of $\fbl_{\mathbb C}[E]$ given in Section \ref{s:fblC}. To this end, we will need first to consider the concept of a free complex vector lattice.

Let us begin by recalling some definitions about complex vector lattices. A \textit{complex vector lattice} $X_\mathbb{C}=X\oplus iX$ is the complexification of a real vector lattice $X$ such that for every $x+iy\in X_\mathbb{C}$ we have that $|x+iy|\in X\subset X_\mathbb{C}$, where $|\cdot|$ stands for the \textit{modulus} function defined in the equation (\ref{eq:modulus}). Typically, the real vector lattice $X$ is assumed to be \textit{uniformly complete} (for instance, in \cite[Section 2.11]{Schaefer-book} or \cite[Section 2.2]{Meyer-Nieberg}) to ensure that the modulus is well-defined. Another equivalent way to define complex vector lattices may be found in \cite{MW,Vuza}, where an axiomatic definition of the modulus mapping on a vector lattice is given. We thank T. Oikhberg for bringing the latter reference to our attention. 

By a \textit{complex vector sublattice} $Y$ of $X_\mathbb{C}$ we mean a conjugation invariant (that is, $x-iy\in Y$ whenever $x+iy\in Y$) complex vector subspace such that $|z|\in Y$ whenever $z\in Y$. A $\mathbb{C}$-linear operator $T:X_\mathbb{C}\to Y_\mathbb{C}$ is said to be a \textit{lattice homomorphism} if it is the complexification of a lattice homomorphism $S:X\to Y$ or, equivalently, if $T$ preserves the modulus of the elements in $X_\mathbb{C}$.

Let $X_\mathbb{C}$ be a complex vector lattice and let $A$ be a non-empty subset of $X_\mathbb{C}$. The \textit{complex vector sublattice generated by $A$} in $X_\mathbb{C}$, which is represented by $\text{lat}_\mathbb{C}(A)$, is the smallest  complex vector sublattice of $X_\mathbb{C}$ which contains $A$. In the real case, we have a useful description of the elements of the sublattice generated by an arbitrary subset: every member of $\text{lat}(A)$ is a lattice-linear combination of a finite number of elements of $A$ \cite[Lemma 5.63]{AA-book}. The upcoming remark provides a description of the elements of $lat_\mathbb{C}(A)$.

\begin{rem}
Let $A$ be a non-empty subset of a complex vector lattice $X_\mathbb{C}=X\oplus iX$. First, we put 
$$
E_1=\text{lat}(\re(A)\cup \im(A))\subset X \quad \text{and}\quad F_1=\{|(x_1,x_2)|\::\: x_1,x_2\in E_1\}.
$$
Given $n\in\mathbb{N}$, such that $n\geq 2$, we define 
$$
E_n=\text{lat}(E_{n-1}\cup F_{n-1}) \quad \text{and}\quad F_n=\{|(x_1,x_2)|\::\: x_1,x_2\in E_n\}.
$$
In this way, we obtain an increasing sequence $\{E_n\}_{n=1}^\infty$ of sublattices of $X$. It is straightforward to check that $E=\cup_{n=1}^\infty E_n$ is a sublattice of $X$ and 
$$
\text{lat}_\mathbb{C}(A)=E\oplus iE.
$$
\end{rem}

We can define an analogous concept of \textit{free vector lattice over a set} (see \cite[Definition 3.1]{dePW}) in the complex setting.

\begin{defn}
If $A$ is any non-empty set, \textit{the free complex vector lattice} over $A$ is a pair $(\fvl_\mathbb{C}(A),\iota)$, where $\fvl_\mathbb{C}(A)$ is a complex vector lattice and $\iota:A\to \fvl_\mathbb{C}(A)$ is a map, with the following universal property: for any complex vector lattice $V_\mathbb{C}$ and any map $T:A\to V_\mathbb{C}$, there exists a unique complex vector lattice homomorphism $\hat{T}:\fvl_\mathbb{C}(A)\to V_\mathbb{C}$ such that $\hat{T}\circ \iota= T$, i.e the following diagram commutes
$$
	\xymatrix{\fvl_\mathbb{C}(A)\ar@{-->}[rrd]^{\hat{T}}\\
	A\ar[u]^{\iota}\ar[rr]^T&&V_\mathbb{C}.}
$$
\end{defn}

It is not difficult to see that if such an object exists, then it is esentially unique up to complex vector lattice isomorphism and this justifies why we refer to it as ``the" free complex vector lattice over $A$. The following proposition ensures the existence of this object. The proof makes use of the free (real) vector lattice over a set $S$, which is denoted $\fvl[S]$ (cf. \cite{Baker,Bleier}).

\begin{prop}
For any non-empty set $A$, $\fvl_\mathbb{C}(A)$ exists.
\end{prop}
\begin{proof}
Given $(x,y)\in A\times A$, we define $\eta_{(x,y)}(f)=f(x,y)$ for every $f\in \mathbb{R}^{A\times A}$. We recall that $\fvl(A\times A)$ is the sublattice generated by $\{\eta_{(x,y)}\::\: (x,y)\in A\times A\}$ in $\mathbb{R}^{A\times A}$ (see \cite[Theorem 3.6]{dePW}). 

Let $V_\mathbb{C}=V\oplus iV$ be a complex vector lattice and $T:A\to V_\mathbb{C}$ a map. Let us consider the following map:
$$
\begin{array}{cccl}
		S: &  A\times A & \longrightarrow & V  \\
		& (x,y)&\longmapsto & \frac{1}{2}\bigl(\mathfrak{Re}Tx+ \mathfrak{Re}Ty -\mathfrak{Im}Tx+ \mathfrak{Im}Ty\bigr).
\end{array}
$$	
By the universal property of $\fvl(A\times A)$, there is a unique lattice homomorphism $\hat{S}:\fvl(A\times A)\to V$ such that $\hat{S}\circ \eta=S$. 

Now, we define a function $\iota:A\to \fvl(A\times A)_\mathbb{C}$ by 
$$
\iota(x)=\eta_{(x,x)}+i\eta_{(-x,x)}
$$ 
for every $x\in A$. The complex lattice homomorphism $\hat{S}_\mathbb{C}:\fvl(A\times A)_\mathbb{C}\to V_\mathbb{C}$ given by $\hat{S}_\mathbb{C}(f+ig)=\hat{S}f+i\hat{S}g$, for every $f,g\in \fvl(A\times A)$, extends the map $T$. Indeed, given $x\in A$, we have
$$
\hat{S}_\mathbb{C}\circ \iota(x)=\hat{S}\eta_{(x,x)}+i\hat{S}\eta_{(-x,x)}=S(x,x)+iS(-x,x)=\mathfrak{Re}Tx+i\mathfrak{Im}Tx=Tx.
$$
Therefore, the complex vector sublattice generated by $\{\iota(x)\::\: x\in A\}$ in $\fvl(A\times A)_{\mathbb C}$ is $\fvl_\mathbb{C}(A)$. Observe that the later condition guarantees the uniqueness of the complex lattice homomorphism which extends $T$. 
\end{proof}

We can also provide a complex version of the notion of \textit{free Banach lattice over a set} introduced by B. De Pagter and A. Wickstead \cite[Definition 4.1]{dePW}.

\begin{defn}
The \textit{free complex Banach lattice} over a non-empty \textit{set} $A$ is a pair $(\fbl_\mathbb{C}(A),\delta_A)$, where $\fbl_\mathbb{C}(A)$ is a complex Banach lattice and $\delta_A:A\to \fbl_\mathbb{C}(A)$ is a bounded map, with the property that for any complex Banach lattice $X_\mathbb{C}$ and any bounded map $T:A\to X_\mathbb{C}$ there is a unique complex vector lattice homomorphism $\hat{T}:\fbl_\mathbb{C}(A)\to X_\mathbb{C}$ such that $\hat{T}\circ \delta_A=T$ and $\|\hat{T}\|=\sup\{\|T(a)\|\,:\, a\in A\}$, i.e. the following diagram commutes
$$
	\xymatrix{\fbl_\mathbb{C}(A)\ar@{-->}[rrd]^{\hat{T}}\\
	A\ar[u]^{\delta_A}\ar[rr]^T&&X_\mathbb{C}}
$$
As usual, if such an object exists, then it is essentially unique up to lattice isometric isomorphism.
\end{defn}

The existence of the free Banach lattice over a set was proved in \cite[Theorem 4.7]{dePW} (the real version of the concept defined above). V. Troitsky found a simpler proof of this fact in \cite{Troitsky}, which is not difficult to adjust to the complex case. We omit the proof of the following proposition on account of it can be proved in a very similar fashion to \cite[Theorem 2.1]{Troitsky}.

\begin{prop}
Let $A$ be a non-empty set and let $\fvl_\mathbb{C}(A)=F\oplus iF$ be the free complex vector lattice over $A$. There exists a maximal lattice seminorm $\nu$ on $F$ with $\nu(|\iota(a)|)\leq 1$ for all $a\in A$. It is a lattice norm and the completion of $\fvl_\mathbb{C}(A)$ respect to $\nu(|\cdot |)$ is $\fbl_\mathbb{C}(A)$.
\end{prop}

The existence of $\fbl_\mathbb{C}(A)$ can also be deduced readily from the existence of the free complex Banach lattice generated by a complex Banach space because $\fbl_\mathbb{C}[\ell_1(A)_\mathbb{C}]$ turns out to be $\fbl_\mathbb{C} (A)$   (compare with \cite[Corollary 2.9]{ART}).

In \cite{Troitsky}, V. Troitsky also provides a description of the free Banach lattice over a real Banach space. The next proposition is an adaptation of Troitsky's construction \cite[Theorem 3.1]{Troitsky} to the complex case and it can be proved in a very similar fashion.
\begin{prop}
Let $E$ be a complex Banach space. Let $L_\mathbb{C}=L\oplus iL$ the complex vector sublattice of $\mathbb{R}^{(E_\mathbb{R})^*}\oplus i\mathbb{R}^{(E_\mathbb{R})^*}$ generated by $\{\delta_E(x)\::\: x\in E\}$. There exists a maximal lattice seminorm $\nu$ on $L$ such that $\nu (|\delta_E(x)|) \leq \|x\|$ for every $x\in L$.  The function $\nu$ is a lattice norm and the norm completion of $L_\mathbb{C}$ respect to $\nu(|\cdot |)$ is $\fbl_\mathbb{C}[E]$.
\end{prop}

\section{Spectra}\label{s:spectra}

Given an endomorphism on a complex Banach space $T:E\rightarrow E$, let us denote $\overline T:\fbl_\mathbb C[E]\rightarrow \fbl_\mathbb C[E]$ the unique lattice homomorphism given by the universal property which makes the following diagram commutative
$$
	\xymatrix{\fbl_\mathbb C[E]\ar[rr]^{\overline T}&&\fbl_\mathbb C[E]\\
	E\ar[u]_{\delta_E}\ar^{T}[rr]&& E\ar[u]_{\delta_E}}
$$
and also satisfies $\|\overline T\|=\|T\|$. Note that in this way we can associate a lattice homomorphism to each bounded linear operator. Our aim in this section is to collect some observations concerning spectral theory via this correspondence. 

As usual for an operator $T:X\rightarrow X$ on a complex Banach space we denote its spectrum $\sigma(T)$ as the (non-empty) compact set consisting of those $\lambda\in\mathbb C$ such that $\lambda-T$ is not invertible (in other words, $\lambda-T$ is not a surjective isomorphism). Let $r(T)=\max\{|\lambda|:\lambda\in\sigma(T)\}$ denote the corresponding spectral radius. As usual $\sigma_p(T)$ denotes the set of eigenvalues (or point spectrum) of $T$ and $$
\begin{array}{ccl}
    \sigma_a(T)&=&\{\lambda\in\sigma(T): \lambda-T\textrm{ is not bounded below}\},  \\
    \sigma_r(T)&=&\{\lambda\in\sigma(T)\backslash \sigma_p(T): \overline{(\lambda-T)(X)}\neq X\}, \\
    \sigma_c(T)&=&\{\lambda\in\sigma(T)\backslash \sigma_p(T): \overline{(\lambda-T)(X)}= X\}
\end{array} 
$$
denote respectively the approximate point spectrum, the residual spectrum and the continuous spectrum. Recall that 
$$
\sigma_p(T)\subseteq \sigma_a(T)\subseteq \sigma(T)\quad\text{and}\quad\sigma(T)=\sigma_p(T)\cup \sigma_r(T)\cup \sigma_c(T),
$$ 
the latter being a disjoint union.

\begin{prop}\label{p:spectra}
Let $T:E\rightarrow E$ be a bounded linear operator and $\overline T:\fbl_\mathbb C[E]\rightarrow \fbl_\mathbb C[E]$ the associated lattice homomorphism.
\begin{enumerate}
\item $\sigma_a(T)\subset \sigma_a(\overline T)$ and $\sigma_p(T)\subset \sigma_p(\overline T)$. 
\item $0\in\sigma(T)$ if and only if $0\in \sigma(\overline T)$.
\item The spectral radii satisfy $r(\overline T)=r(T)$.
\end{enumerate}
\end{prop}

\begin{proof}
(1) Let $T:E\rightarrow E$, and suppose $\lambda\notin \sigma_a(\overline T)$. Thus, there is $\alpha>0$ such that $\|(\lambda-\overline T)f\|\geq\alpha\|f\|$ for every $f\in \fbl_\mathbb C[E]$. Hence, we have that for $z\in E$ 
$$
\|(\lambda-T)z\|=\|(\lambda-\overline T)\delta_E(z)\|\geq \alpha \|\delta_E(z)\|=\alpha\|z\|.
$$
Thus, $\lambda\notin\sigma_a(T)$.

Similarly, the statement about point spectra follows from observing that if $Tx=\lambda z$, then $\overline T \delta_E(z)= \delta_E(\lambda z)=\lambda\delta_E(z)$.

(2) Suppose $0\notin\sigma(T)$, that is $T:E\rightarrow E$ is an isomorphism. Then there is $T^{-1}:E\rightarrow E$ such that $T T^{-1}=T^{-1}T=\text{Id}_E$. Since $\overline{\text{Id}_E}=\text{Id}_{\fbl_\mathbb C[E]}$, it is clear that in this case, $\overline{T^{-1}}=\overline T^{-1}$, so that $\overline T$ is an isomorphism and $0\notin \sigma(\overline T)$.

Conversely, suppose $\overline T$ is an isomorphism. Note that $\overline T$ maps the range of $\delta_E$ onto itself. Indeed, we have that $\overline T(\delta_E(E))=\delta_E T(E)\subset \delta_E(E)$ and so $\overline T(\delta_E(E))$ is a closed subspace of $\fbl_\mathbb C[E]$ because $\overline T$ is an isomorphism. If $F=\overline T(\delta_E(E))\subsetneq \delta_E(E)$, then we can find a nonzero $z^*\in E^*$ which is identically zero on $\delta_E^{-1}(F)$. Since $\overline T$ is a lattice homomorphism, its range would be contained in the closed sublattice generated by $F$. But $\overline{\text{lat}}\{F\}\subset \{f\in \fbl_\mathbb{C}[E]\::\: f(\re z^*)=0\}$ and hence $\overline T$ would not be onto, which is a contradiction.

Now, we claim that the inverse 
$$
\overline T^{-1}:\fbl_\mathbb C[E]\to \fbl_\mathbb C[E]
$$ 
also maps $\delta_E(E)$ to $\delta_E(E)$. Indeed, suppose that for some $w\in E$ we have $\overline T^{-1}\delta_E(w)=f\in \fbl_\mathbb C[E]\backslash \delta_E(E)$. Since $\overline T(\delta_E(E))=\delta_E(E)$, there is $z\in E$ such that $\overline T\delta_E (z)=\delta_E(w)$. Hence, we would have that $\overline Tf=\delta_E(w)=\overline T\delta_E(z)$, which is a contradiction with the injectivity of $\overline T$. It follows then that 
$$
\delta^{-1}_E\overline T^{-1}\delta_E:E\to E
$$ 
is the inverse of $T$, thus showing that $0\notin \sigma(T)$.

(3) For every $n\in\mathbb N$, we have that $\overline T^n$ is a lattice homomorphism on $\fbl_\mathbb C[E]$ which extends $T^n$. Hence, $\overline {T^n}=\overline T^n$, and in particular
$$
\|\overline T^n\|=\|T^n\|.
$$
It follows from Gelfand's formula for the spectral radius \cite[Theorem 6.12]{AA-book} that
$$
r(\overline T)=\lim_n \|\overline T^n\|^{\frac1n}=\lim_n \|T^n\|^{\frac1n}=r(T).
$$
\end{proof}

\begin{rem}
The following facts, which are known in the real case (see \cite[Section 3]{OTTT}), can be easily extended to the complex setting: 
\begin{enumerate}
    \item $T$ is injective if and only if $\overline{T}$ is injective.
    \item $T$ has dense range if and only if $\overline T$ has dense range.
\end{enumerate}
Taking this into account, the second part of the former proposition can be refined as follows: $0\in \sigma_p(T)$ if and only if $0\in \sigma_p(\overline T)$; $0\in \sigma_c(T)$ if and only if $0\in \sigma_c(\overline T)$; $0\in \sigma_r(T)$ if and only if $0\in \sigma_r(\overline T)$.
\end{rem}

\begin{rem}
In general, we can have $\sigma_p(T)\neq\sigma_p(\overline T)$, $\sigma_a(T)\neq\sigma_a(\overline T)$ or $\sigma(T)\neq\sigma(\overline T)$. Consider the operator $T:\mathbb{C}\to \mathbb{C}$ defined by $Tz=-z$, which has $\sigma(T)=\sigma_a(T)=\sigma_p(T)=\{-1\}$. However, the function $(f+ig):\mathbb{R}^2\to \mathbb{C}$ defined by $(f+ig)(x,y)=\sqrt{x^2+y^2}$ is a fixed point of $\overline{T}$ and then $1\in \sigma_p(\overline T)$.
\end{rem}

In view of the results exposed in Proposition \ref{p:spectra} one may ask whether it is always true that $\sigma(T)$ is contained in $\sigma(\overline{T})$. The remainder of this section will be dedicated to exploring this question. Let us begin by showing that the positive elements of $\sigma(T)$ belong also to $\sigma(\overline{T})$.

\begin{prop}\label{prop:positive-preservation}
If $\lambda\in \sigma(T)\cap [0,+\infty)$, then $\lambda\in \sigma (\overline{T})$. 
\end{prop}
\begin{proof}
Since $\sigma(T)=\sigma_a(T)\cup\sigma_r(T)$, by the first part of Proposition \ref{p:spectra} it is enough to show that $\sigma_r(T)\cap [0,+\infty)\subset \sigma(\overline T)$. Fix any $\lambda\in\sigma_r(T)\cap [0,+\infty)$. Then $\lambda\in \sigma_p(T^*)$ (see \cite[Theorem 6.19]{AA-book}). Take $z^*\in E^*$, $z^*\neq 0$, such that $T^*z^*=\lambda z^*$. For every $f\in \fbl_\mathbb{C}[E]$, by the positive homogeneity of $f$, we have that
$$(\lambda - \overline T)f(\re z^*)=\lambda f(\re z^*) - f(\re T^* z^*)=\lambda f(\re z^*)- f(\lambda \re z^*)=0,$$
so that $\lambda - \overline T$ cannot have dense range and, in particular, $\lambda\in \sigma(\overline T)$.
\end{proof}

Given $\theta\in\mathbb R$, let us denote by $M_\theta$ the operator defined on $E$ by $M_\theta z=e^{i\theta}z$, for every $z\in E$. These operators will play an important role in the proof of Proposition \ref{prop:preserves spectrum}, which is a generalization of the preceding proposition. Before this, we are going to describe the spectrum of the operators $\{M_\theta\}_{\theta\in\mathbb{R}}$.

\begin{prop}
Suppose that $e^{i\theta}$ is a primitive $n$th root of unity. Then $\sigma(\overline{M_\theta})=\{e^{i\frac{2k\pi}{n}}\::\: k=0,1,\ldots, n-1\}$.
\end{prop}
\begin{proof}
First, observe that $\overline{M_\theta}$ is a lattice isometry since $M_\theta$ is an isometry. Thus, $\sigma_a(M_\theta)=\sigma(M_\theta)\subset \partial \mathbb{D}=\{\lambda\in \mathbb C:|\lambda|=1\}$. As $e^{i\theta}$ belongs to $\sigma_p(M_\theta)$, by Proposition 6.1, $e^{i\theta}$ is also in $\sigma(\overline{M_\theta})$. Given that the spectrum of a lattice homomorphism is cyclic (see \cite[Theorem 7.23]{AA-book}), we have that 
$$
\{e^{i\frac{2k\pi}{n}}\::\: k=0,1,\ldots, n-1\}\subset \sigma(\overline{M_\theta}).
$$

Now, let us see that  $\sigma_a(\overline{M_\theta})=\sigma(\overline{M_\theta})\subset \{e^{i\frac{2k\pi}{n}}\::\: k=0,1,\ldots, n-1\}$. To this end, take $\lambda\in \sigma_a(\overline{M_\theta})$. Then, we can find a sequence of complex numbers $(\lambda_k)_{k=1}^\infty$ with $\lambda_k\to \lambda$ and $(f_k)_{k=1}^\infty \subset \fbl_\mathbb{C}[E]$ with $\|f_k\|=1$ for every $k\in\mathbb N$, such that $\|\lambda_k f_k- \overline{M_\theta} f_k\| \to 0$. It is straightforward to check by induction that 
$$
\|\lambda_k^n f_k- \overline{M_\theta}^n f_k\|\to 0
$$ 
as $k\to\infty$. Indeed, let us suppose that $\|\lambda_k^m f_k- \overline{M_\theta}^m f_k\|\to 0$ for some $m\in\mathbb{N}$. For every $k\in\mathbb{N}$ we have that
\begin{eqnarray*}
 \|\lambda_k^{m+1} f_k- \overline{M_\theta}^{m+1} f_k\| &=& \|\lambda_k^{m+1} f_k - \lambda_k^m \overline{M_\theta} f_k+ \lambda_k^m \overline{M_\theta} f_k- \overline{M_\theta}^{m+1} f_k\| \\
 &\leq & |\lambda_k^m|\|\lambda_k f_k -\overline{M_\theta} f_k\|+ \| \overline{M_\theta}\|\|\lambda_k^m f_k- \overline{M_\theta}^m f_k\|,
\end{eqnarray*}
and, by hypothesis, the last term converges to zero.
 Since $\overline{M_\theta}^n=\text{Id}$, we deduce that $\|\lambda_k^n f_k- \overline{M_\theta}^n f_k\|=|\lambda_k^n-1|\to 0$ as $k\to\infty$, that is  $\lambda_k^n\to 1$. By the uniqueness of the limit, we obtain that $\lambda^n=1$, that is, $\lambda\in \{e^{i\frac{2k\pi}{n}}\::\: k=0,1,\ldots, n-1\}$.
\end{proof}

\begin{rem}
If $\theta=2\pi t$, for some irrational number $t$, then $\sigma(\overline{M_\theta})=\partial \mathbb{D}$. Indeed, observe that $\sigma(\overline{M_\theta})\subset \partial \mathbb{D}$ since $\overline{M_\theta}$ is an isometry. On the other hand, since $\sigma(\overline{M_\theta})$ is closed and cyclic, we have that 
$
\partial \mathbb{D}=\overline{\{e^{i n t}\::\:n\in\mathbb{N}\}}\subset \sigma(\overline{M_\theta}).
$
\end{rem}

\begin{prop}\label{prop:preserves spectrum}
Let $T:E\to E$ be a complex operator. If $\lambda\in\sigma(T)$, then $|\lambda|\in \sigma(\overline T)$.
\end{prop}
\begin{proof}
Let us write $\lambda=|\lambda|e^{i\theta}$ for some $\theta\in [0,2\pi)$. Therefore, $|\lambda|\in \sigma(M_{-\theta}T)$, and by Proposition \ref{prop:positive-preservation} we get that 
$$
|\lambda|\in \sigma(\overline{M_{-\theta}T})=\sigma(\overline{M_{-\theta}}\,\overline{T})\subset \sigma(\overline{M_{-\theta}})\sigma(\overline{T}),
$$ 
the latter inclusion following from commutativity of $\overline{M_{-\theta}}$ and $\overline{T}$ \cite[Theorem 11.23]{Rudin}.

Since $\sigma(\overline{M_{-\theta}})\subset \partial \mathbb{D}$ and $\sigma(\overline T)$ is cyclic by \cite[Theorem 7.23]{AA-book}, it follows that $|\lambda|\in\sigma(\overline{T})$.
\end{proof}

We will see next that in the case when $E$ is a Banach lattice and $T$ is a lattice homomorphism, then we always have $\sigma(T)\subset \sigma(\overline T)$. This fact will cover many instance of classical operators arising in the literature:

\begin{example}
The following operators are remarkable examples of complex lattice homomorphisms.
\begin{enumerate}
    \item Let $K$ be a compact Hausdorff space and let $h:K\to K$ a continuous mapping. The \textit{composition operator} $C_h:\mathcal{C}(K)_\mathbb{C}\to\mathcal{C}(K)_\mathbb{C}$ defined by $C_h(f)=f\circ h$. 
    \item If $g\in\mathcal{C}(K)_\mathbb{C}$ is a positive function, the \textit{multiplication operator} $M_g(f)=g f$, for every $f\in \mathcal{C}(K)_\mathbb{C}$
    \item For $1\leq p\leq \infty$, the \textit{backward shift operator} $B:\ell_p\to \ell_p$ defined by
    $$
    B(z_1,z_2,\ldots)=(z_2,z_3,\ldots),
    $$
    and the \textit{forward shift operator} $F:\ell_p\to \ell_p$ defined by
    $$
    F(z_1,z_2,\ldots)=(0,z_1,z_2,\ldots).
    $$
\end{enumerate}
\end{example}

\begin{lem}
If $X_\mathbb{C}$ is a complex Banach lattice, then there exists an ideal $I_\mathbb{C}$ in $\fbl_\mathbb{C}[X_\mathbb{C}]$ such that $\fbl_\mathbb{C}[X_\mathbb{C}]=\delta_{X_\mathbb{C}}(X_\mathbb{C})\oplus I_\mathbb{C}$.
\end{lem}
\begin{proof}
By the universal property of $\fbl_\mathbb{C}[X_\mathbb{C}]$, there is a complex lattice homomorphism $\beta:\fbl_\mathbb{C}[X_\mathbb{C}]\to X_\mathbb{C}$ such that $\beta\delta_{X_\mathbb{C}}=\text{Id}$. It is easy to check that the composition $P=\delta_{X_\mathbb{C}}\beta:\fbl_\mathbb{C}[X_\mathbb{C}]\to \fbl_\mathbb{C}[X_\mathbb{C}]$ defines a projection onto $\delta_{X_\mathbb{C}}(X_\mathbb{C})$. Since $\delta_{X_\mathbb{C}}$ is an isometry, we deduce that $\text{Ker}(P)=\text{Ker}(\beta)$ and note that $\text{Ker}(\beta)$ is an ideal in $\fbl_\mathbb{C}[X_\mathbb{C}]$ as $\beta$ is a complex lattice homomorphism.
\end{proof}

\begin{prop}\label{p:latticehomo}
If $T:X_\mathbb{C}\to X_\mathbb{C}$ is a lattice homomorphism, then $\sigma(T)\subset \sigma(\overline{T})$. 
\end{prop}
\begin{proof}
For simplicity, throughout this proof we shall write $Z=X_\mathbb{C}$.  By the previous lemma we know that $\fbl_\mathbb{C}[Z]$ can be decomposed into a direct sum $\fbl_\mathbb{C}[Z]=\delta_{Z}(Z)\oplus I_\mathbb{C}$, where $I_\mathbb{C}=\text{Ker}(\beta)$. Since $\beta\delta_Z=\text{Id}$, by the definition of $\delta_Z$, we have that 
$$
(x_1,x_2)=\beta\delta_Z(x_1,x_2)=(\beta\delta_{Z_\mathbb{R}}(x_1,x_2),\beta\delta_{Z_\mathbb{R}}(x_2,-x_1)),
$$ 
for every $(x_1,x_2)\in Z$, so $\beta \delta_{Z_\mathbb{R}}(x_1,x_2)=x_1$.

Let us see that $\overline{T}(f+ig)\in \text{Ker}(\beta)=I_\mathbb{C}$ whenever $f+ig\in \text{Ker}(\beta)=I_\mathbb{C}$. This is equivalent to the fact that $\overline T f\in I$ whenever $f\in I$. Fix any $f\in I$. Since $f\in \fbl[Z_\mathbb{R}]$, there is a sequence $(f_n)_{n=1}^\infty$ in $\fvl[Z_\mathbb{R}]$ such that $\|f_n-f\|_{\fbl_\mathbb{C}[Z_\mathbb{R}]}\to 0$. For each $n\in\mathbb{N}$, we shall write
$$
f_n=\bigvee_{k=1}^{m(n)} \delta_{Z_\mathbb{R}}\left(x_{1,k}^{(n)},x_{2,k}^{(n)}\right)- \bigvee_{j=1}^{m(n)} \delta_{Z_\mathbb{R}}\left(y_{1,j}^{(n)},y_{2,j}^{(n)}\right).
$$

By the continuity of $\beta$ and $\overline T$, we have that the sequence $\beta \overline T f_n$ converges to $\beta \overline T f$. The next identity will enable us to show that the limit $\beta \overline T f$ is equal to zero:
\begin{eqnarray*}
 \beta \overline T f_n &=&  \beta \overline T \left(\bigvee_{k=1}^{m(n)} \delta_{Z_\mathbb{R}}\left(x_{1,k}^{(n)},x_{2,k}^{(n)}\right)- \bigvee_{j=1}^{m(n)} \delta_{Z_\mathbb{R}}\left(y_{1,j}^{(n)},y_{2,j}^{(n)}\right)\right) \\
&=& \beta \left(\bigvee_{k=1}^{m(n)} \delta_{Z_\mathbb{R}}\left(T x_{1,k}^{(n)},T x_{2,k}^{(n)}\right)- \bigvee_{j=1}^{m(n)} \delta_{Z_\mathbb{R}}\left(T y_{1,j}^{(n)},T y_{2,j}^{(n)}\right)\right) \\
&=& \bigvee_{k=1}^{m(n)} T x_{1,k}^{(n)} - \bigvee_{j=1}^{m(n)} T y_{1,j}^{(n)} = T\left(\bigvee_{k=1}^{m(n)} x_{1,k}^{(n)} - \bigvee_{j=1}^{m(n)} y_{1,j}^{(n)}  \right)=T (\beta f_n).
\end{eqnarray*}
Thus, by the continuity of $T$ we obtain that $(\beta \overline T f_n)_{n=1}^\infty$ converges to $T \beta f$, which is zero, since we are assuming that $f\in I=\text{Ker}(\beta)$. By the uniqueness of the limit, $\beta \overline T f=0$, so $\overline T f\in I$.
Finally, given that $I_\mathbb{C}$ and $\delta_{Z}(Z)$ are $\overline T$-invariant subspaces, it can be easily verified that (see \cite[Section 6.1, Ex. 17]{AA-book})
$$
\sigma(\overline{T})=\sigma(T)\cup \sigma\bigl(\left.\overline T\right|_{I_\mathbb{C}}\bigr).
$$
\end{proof}

\begin{rem}
The previous proposition remains true for operators which are \textit{similar} to lattice homomorphisms: an operator $T$ on a complex Banach space $X$ is said to be similar to a lattice homomorphism if there exist a complex Banach lattice $Y_\mathbb{C}$ and an invertible operator $S:X\to Y_\mathbb{C}$ such that $STS^{-1}$ is a lattice homomorphism.
\end{rem}

The following remark suggests that Proposition \ref{p:latticehomo} might not hold for $T$ being an arbitrary operator.

\begin{rem}
Let $E$ be an uncomplemented subspace of $\ell_1$ which is isomorphic to $\ell_1$ \cite{Bourgain81}. Let $T:\ell_1 \to \ell_1$ be an isomorphism onto $E$. We have that $0\in \sigma_a(\overline T)$ even though $0\in \sigma(T)\backslash\sigma_a(T)$. Indeed, suppose otherwise that $\overline T$ is bounded below. Let us denote by $S:\ell_1\to E$ the operator defined by $Sx=Tx$ for every $x\in \ell_1$, so that $\iota S=T$, where $\iota$ stands for the canonical inclusion of $E$ into $\ell_1$. Thus, $\overline T=\overline \iota \overline S$ and since $\overline S$ is a lattice isomorphism of $\fbl_\mathbb{C}[\ell_1]$ onto $\fbl_\mathbb{C}[E]$ and we are assuming that $\overline T$ is bounded below, it follows that $\overline \iota$ is a lattice embedding $\fbl_\mathbb{C}[E]$ into $\fbl_\mathbb{C}[\ell_1]$. This implies that any $R:E\to L_1(\mu)$ extends to $\hat{R}:\ell_1\to L_1(\mu)$ with $\|\hat R\|=\|R\|$ (see \cite[Section 3]{OTTT}). In particular, $S^{-1}:E\to \ell_1$ extends to $\widehat{S^{-1}}:\ell_1\to\ell_1$ and then we could define a projection $P=\iota S\widehat{S^{-1}}$ of $\ell_1$ onto $E$. This is a contradiction.
\end{rem}

\section*{Acknowledgements}
Research supported by Agencia Estatal de Investigaci\'on, Ministerio de Ciencia e Innovaci\'on under grants PID2020-116398GB-I00 and CEX2019-000904-S funded by: MCIN/AEI/ 10.13039/501100011033.  The first-named author benefited from an FPU Grant FPU20/03334 from the Ministerio de Universidades.

\def\cprime{$'$}

\providecommand{\MR}{\relax\ifhmode\unskip\space\fi MR }

\end{document}